\documentclass[12pt,reqno]{amsart}

\usepackage{amsthm,amssymb,amstext,amscd,euscript,mathrsfs,amsfonts,amsbsy,amsxtra,latexsym,amsmath}
\usepackage{fullpage}
\usepackage[english]{babel}
\usepackage[latin1]{inputenc}
\usepackage{verbatim}
\usepackage{graphicx} 
\usepackage[all]{xy} 
\usepackage{enumitem}
\usepackage{bbm}
\usepackage{marvosym}
\numberwithin{figure}{section} 
\allowdisplaybreaks

\newcommand{\field}[1]{\mathbb{#1}} 

\newcommand{\Z}{\field{Z}} 

\newcommand{\C}{\field{C}}
\newcommand{\Q}{\field{Q}}

\numberwithin{equation}{section}
\newtheorem{theorem}{\textbf{Theorem}}
\numberwithin{theorem}{section}
\newtheorem{corollary}[theorem]{\textbf{Corollary}}
\newtheorem{lemma}[theorem]{\textbf{Lemma}}
\newtheorem{proposition}[theorem]{\textbf{Proposition}}
\theoremstyle{definition}
\newtheorem{definition}[theorem]{Definition}
\newtheorem{remark}{Remark}[section]

\newtheorem{claim}[theorem]{Claim}

\newcommand{\bea}{\begin{eqnarray}} 
\newcommand{\eea}{\end{eqnarray}} 
\newcommand{\be}{\begin{equation}} 
\newcommand{\ee}{\end{equation}} 
\newcommand{\benn}{\begin{equation*}} 
\newcommand{\eenn}{\end{equation*}}

\title[short title]{Stringy Hirzebruch classes of Weierstrass fibrations}
\author{James Fullwood}
\author{Mark van Hoeij$\mbox{}^*$}\thanks{$\mbox{}^*$ supported by NSF grant 1618657}
\address{School of Mathematical Sciences\\Shanghai Jiao Tong University\\ 800 Dongchuan Road, Shanghai, China}
\email{fullwood@maths.hku.hk}
\address{Mathematics Department\\Florida State University\\Tallahassee, FL 32306, U.S.A.}
\email{hoeij@math.fsu.edu}

\begin{document}

\maketitle

\begin{abstract}
A Weierstrass fibration is an elliptic fibration $Y\to B$ whose total space $Y$ may be given by a global Weierstrass equation in a $\mathbb{P}^2$-bundle over $B$. In this note, we compute stringy Hirzebruch classes of singular Weierstrass fibrations associated with constructing non-Abelian gauge theories in $F$-theory. For each Weierstrass fibration $Y\to B$ we then derive a generating function $\chi^{\text{str}}_y(Y;t)$, whose degree-$d$ coefficient encodes the stringy $\chi_y$-genus of $Y\to B$ over an unspecified base of dimension $d$, solely in terms of invariants of the base. To facilitate our computations, we prove a formula for general characteristic classes of blowups along (possibly singular) complete intersections. 
\end{abstract}

\tableofcontents

\section{Introduction}\label{intro}
Let $X$ be a smooth complex variety, $\mathscr{E}\to X$ a holomorphic vector bundle, and let $\widetilde{T}_y(\mathscr{E})$ be the cohomological characteristic class given by
\[
\widetilde{T}_y(\mathscr{E})=\sum_{q\geq 0}\text{ch}(\Lambda^q\mathscr{E}^{\vee})\text{td}(\mathscr{E})y^q\in H^*(X)\otimes \Q[y],
\]
which we refer to as the \emph{Hirzebruch class} of $\mathscr{E}$. The homology class 
\[
\widetilde{T}_y(TX)\cap [X]\in H_*(X)\otimes \Q[y]
\]
will then be referred to as the Hirzebruch class of $X$, which we denote by $\widetilde{T}_y(X)$. The zero-dimensional component $\int\widetilde{T}_y(X)$ of $\widetilde{T}_y(X)$ is denoted by $\chi_y(X)$, which is referred to as \emph{Hirzebruch's} $\chi_y$-\emph{genus}. By the Hirzebruch-Riemann-Roch theorem,
\begin{equation}\label{hnr}
\chi_y(X)=\sum_{q\geq 0}\left(\sum_{i\geq 0}(-1)^i\text{dim}_{\C}H^i(\Lambda^qT^*X)\right)y^q\in \Z[y],
\end{equation}
so that $\chi_y(X)$ encodes linear relations among the Hodge numbers 
\[
h^{p,q}(X)=\text{dim}_{\C}H^p(\Lambda^qT^*X),
\]
and moreover, evaluating $\chi_y(X)$ at $y=-1,0,1$ yields the topological Euler characteristic $\chi(X)$, the arithmetic genus $\chi_a(X)$, and the signature $\sigma(X)$ respectively.

While the Hirzebruch class $\widetilde{T}_y$ encodes fundamental invariants of a smooth variety, its associated power series, namely
\[
\widetilde{Q}(z)=\frac{z(1+ye^{-z})}{1-e^{-z}},
\]
admits the somewhat undesirable feature of being un-normalized, i.e.,
\[
\widetilde{Q}(0)=1+y\neq 1.
\] 
However, the power series
\begin{equation}\label{npsd}
Q(z)=\widetilde{Q}(z(1+y))\cdot (1+y)^{-1}=\frac{z(1+y)}{1-e^{-z(1+y)}}-yz
\end{equation}
is such that $Q(0)=1$, and equation \eqref{npsd} further implies that the characteristic class $T_y$ associated with $Q(z)$ is a normalized class which agrees with $\widetilde{T}_y$ in top degree. As such, the characteristic class $T_y$ will be referred to as the \emph{normalized Hirzebruch class}, and if $X$ is a smooth variety then the homology class
\[
T_y(TX)\cap [X] \in H_*(X)\otimes \Q[y]
\]
will be referred to as the normalized Hirzebruch class of $X$, which will be denoted $T_y(X)$. Since $T_y$ and $\widetilde{T}_y$ agree in top degree we have  
\begin{equation}\label{d0}
\int_XT_y(X)=\int_X\widetilde{T}_y(X)=\chi_y(X),
\end{equation}
and evaluating $T_y(X)$ at $y=-1,0,1$ yields
\[
T_{-1}(X)=c(X), \quad T_0(X)=\text{td}(X), \quad \text{and} \quad T_1(X)=L(X),
\]
thus the normalized Hirzebruch class unifies the notions of Chern, Todd, and $L$-class while retaining $\chi_y(X)$ in dimension zero. This unifying aspect of the normalized Hirzebruch class makes it more fundamental in some sense than its un-normalized counterpart, and when the context is clear we will often refer to the normalized Hirzebruch class simply as the Hirzebruch class. 

There are various extensions of Hirzebruch classes to singular $X$ \cite{MHC}, and it is the `stringy' notion with which we concern ourselves. Given an invariant for smooth varieties, an extension of the invariant to singular varieties is often referred to as \emph{stringy} if it is invariant with respect to the notion of $K$-equivalence in birational geometry. In particular, stringy invariants are preserved under maps which are crepant, i.e., proper birational maps $f:Y\to X$ such that $K_Y=f^*K_X$. Stringy invariants require restrictions on the allowed singularities for their definition, and for $X$ with at worst Gorenstein canonical singularities, there exists \emph{stringy Hirzebruch classes} $\widetilde{T}_y^{\text{str}}(X)$ and $T_y^{\text{str}}(X)$ which agree in dimension-zero, and this dimension-zero component common to both, which we denote by $\chi^{\text{str}}_y(X)$, will be referred to as the \emph{stringy} $\chi_y$-\emph{genus} of $X$. We note that while Hirzebruch's $\chi_y$-genus is a polynomial in $y$, $\chi_y^{\text{str}}(X)$ is a priori a rational function in $y$. The stringy $\chi_y$-genus is invariant with respect to crepant maps, and if $X$ admits a crepant resolution $\rho:Z\to X$, then 
\begin{equation}\label{crp}
\rho_*\widetilde{T}_y(Z)=\widetilde{T}_y^{\text{str}}(X) \quad \text{and} \quad \rho_*T_y(Z)=T_y^{\text{str}}(X),
\end{equation}
and in such a case $\chi_y^{\text{str}}(X)$ is a polynomial whose coefficients yield linear relations for the \emph{stringy Hodge numbers} of $X$ as defined by Batyrev \cite{SHNB}. Evaluating $T^{\text{str}}_y(X)$ at $y=-1$  yields the \emph{stringy Chern class} of $X$, and as such, evaluating $\chi^{\text{str}}_y(X)$ at $y=-1$ yields the \emph{stringy Euler characteristic} of $X$ \cite{SHNB}\cite{SCCA}\cite{SCCdF}. Evaluating $\chi^{\text{str}}_y(X)$ at $y=0,1$ then yields stringy versions of the arithmetic genus and signature. In \S\ref{SHC} we review in further detail the general theory for stringy characteristic classes.

In this note, we compute stringy Hirzebruch classes of elliptic fibrations $Y\to B$ whose total space $Y$ may be embedded as a hypersurface in a $\mathbb{P}^2$-bundle given by a global Weierstrass equation
\begin{equation}\label{we}
Y:(y^2z=x^3+fxz^2+gz^3)\subset \mathbb{P}(\mathscr{E}),
\end{equation}
where $f$ and $g$ are sections of tensor powers of a line bundle $\mathscr{L}\to B$, referred to as \emph{the fundamental line bundle} of $Y\to B$. We refer to such elliptic fibrations as \emph{Weierstrass fibrations}. The rank-3 vector bundle $\mathscr{E}\to B$ which is projectivized to construct the ambient space of $Y$ is then given by $\mathscr{E}=\mathscr{O}_B\oplus \mathscr{L}^2\oplus \mathscr{L}^3$, where $\mathscr{L}$ is the fundamental line bundle.

The Weierstrass fibrations we consider consists of 14 different families which are relevant for constructing super-symmetric string vacua whose associated field theories admit non-Abelian gauge symmetries \cite{EJK}. In $F$-theory, a gauge group $\mathcal{G}_Y$ is associated with a singular Weierstrass fibration $Y\to B$ according to the type singular fibers which appear upon a resolution of singularities and its Mordell-Weil group of rational sections. In order to avoid complications associated with singular compactifications of string vacua, crepant resolutions have been constructed in the $F$-theory literature for a number of Weierstrass fibrations \cite{EYSU5}\cite{TFoS}\cite{EJK}, and in such a case we exploit property \eqref{crp} to compute their stringy Hirzebruch classes. There are often multiple crepant resolutions of Weierstrass fibrations, which are connected by a network of flop transitions. As the particular resolution one uses among a network of choices is often irrelevant from the physics perspective, the physically relevant invariants one computes via such a resolution should be seen as \emph{stringy invariants} of the \emph{singular} Weierstrass fibration one is resolving, rather than invariants of the resolved geometry. In any case, as such crepant resolutions are constructed by successively blowing up the ambient space $\mathbb{P}(\mathscr{E})$ of a Weierstrass fibration along complete intersections, in \S\ref{CSCC} we employ intersection-theoretic techniques along with a blowup formula of Aluffi (\cite{ACCB}, Theorem~1.2) to derive a formula for general characteristic classes of blowups along (possibly singular) complete intersections, along with an associated pushforward formula for computing $\rho_*$ as in \eqref{crp}.

As an elliptic fibration may be characterized by its configuration of singular fibers, it is often possible to compute invariants of an elliptic fibration $Y\to B$ relative to an unspecified base $B$. For example, since only the singular fibers of an elliptic fibration $Y\to B$ contribute its topological Euler characteristic $\chi(Y)$, one may stratify the discriminant of the fibration into strata over which the topological type of the fiber is constant, and then give a formula for $\chi(Y)$ in terms of the combinatorics of the fiber structure and the geometry of the strata of the discriminant. Such methods were employed for example in \cite{GMEC3F}, where formulas for the Euler characteristic of elliptic 3-folds were given in terms of a certain representation of a Lie group determined by the fiber structure of the fibration.

More recently, techniques from intersection theory have been used to compute the Euler characteristic of elliptic fibrations over an unspecified base of arbitrary dimension \cite{AE1}\cite{AE2}\cite{EFY}\cite{F}, generalizing formulas found in the physics literature for elliptic 3- and 4-folds \cite{SVW}\cite{KLRY}. In particular, if $Y\to B$ is a smooth Weierstrass fibration with fundamental line bundle $\mathscr{L}\to B$, its topological Euler characteristic was computed in \cite{AE1} as
\[
\chi(Y)=12c_1(\mathscr{L})\sum_{i=0}^{\text{dim}(B)-1}c_i(B)(-6c_1(\mathscr{L}))^{\text{dim}(B)-1-i},
\]
so that a formula  for $\chi(Y)$ may be given in terms of invariants of an unspecified base of arbitrary dimension. Such formulas were referred to as `Sethi-Vafa-Witten formulas' in \cite{AE1}, and similar formulas for stringy Euler characteristics of singular Weierstrass fibrations and elliptic fibrations not in Weierstrass form have appeared in \cite{FvH}\cite{EJK}\cite{AE2}\cite{EFY}. In a similar vein, given a singular Weierstrass fibration $Y\to B$, in \S\ref{GF} we derive a generating function $\chi_y^{\text{str}}(Y;t)=\sum_na_nt^n$, where $a_n$ is a formula for $\chi_y^{\text{str}}(Y)$ over an unspecified base $B$ of dimension $n$, solely in terms of invariants of $B$. Such formulas can then be used to derive relations between the stringy Hodge numbers of $Y$, via `stringy' Hirzebruch-Riemann Roch. For example, in \S\ref{GF} we show that if $Y$ is a Weierstrass fibration with gauge group $\mathcal{G}_Y=\text{SO}(6)$, then
\[
\chi_y^{\text{str}}(Y;t)=\left(1-5y+{\frac { \left( y+1 \right)  \left(  \left( e^{t(1+y)L}+4 \right) y-e^{t(1+y)L} \right) }{{e^{2t(1+y)L}}+y}}\right)\exp\left(R(t)\odot \left(\frac{-tC'(t)}{C(t)}\right)\right),
\]
where $L$ is the first Chern class of the fundamental line bundle $\mathscr{L}\to B$ of $Y$, $R(t)=\ln (Q(t))$ for $Q(t)$ the characteristic power series of the normalized Hirzebruch class as given in equation \eqref{npsd}, $C(t)$ is the formal power series $C(t)=\sum_{i=0}^{\infty}(-1)^ic_it^i$ (where the $c_i$s are interpreted as the Chern classes of $B$), and $\odot$ denotes the Hadamard product of power series. 
\\ 

\noindent \emph{Notation and conventions.} A variety will always be assumed to be a reduced and separated scheme of finite type over the complex numbers $\C$. Given a holomorphic vector bundle $\mathscr{E}\to B$ over a variety $B$, $\mathbb{P}(\mathscr{E})\to B$ will always be taken to denote the associated projective bundle of lines in $\mathscr{E}$, and given a line bundle $\mathscr{L}\to B$, its $m$-th tensor power will be denoted $\mathscr{L}^m\to B$.    

\section{Stringy characteristic classes}\label{SHC}
Let $X$ be a variety, and let $A_*X$ denote the group of algebraic cycles in $X$ modulo rational equivalence. Assume $\mathfrak{C}$ is a characteristic class associated with holomorphic vector bundles $\mathscr{E}\to X$ such that for every $\alpha\in A_*X$ 
\[
\mathfrak{C}(\mathscr{E})\cap \alpha \in A_*X\otimes R,
\]
where $R$ is a commutative ring with unity. If $X$ is smooth, the class $\mathfrak{C}(TX)\cap [X]$ will be denoted simply by $\mathfrak{C}(X)$. If $X$ is singular, we may define a `stringy' extension $\mathfrak{C}^{\text{str}}(X)$ of $\mathfrak{C}$ provided we suitably restrict the singularities of $X$. As such, we now define what it means for $X$ to have at worst Gorenstein canonical singularities. 

\begin{definition}\label{D1}
Let $X$ be a normal irreducible variety over $\C$. We will say that $X$ admits at worst \emph{Gorenstein canonical singularities} if and only if the canonical divisor $K_X$ is Cartier, and for any resolution of singularities $\rho:Z\to X$ such that the exceptional locus is a smooth normal crossing divisor with irreducible components $\{D_i\}_{i\in I}$, the discrepancy divisor 
\[
K_Z-\rho^*K_X=\sum_{i\in I}a_iD_i
\]
is such that $a_i\geq 0$. Such a resolution $\rho:Z\to X$ will be referred to as a \emph{log resolution} of $X$ (log resolutions always exist).
\end{definition}

\begin{definition}\label{D2}
Let $f:Y\to X$ be a proper birational map between varieties. We will refer to $f$ as \emph{crepant} if and only if $K_Y=f^*K_X$, and if $f$ is also a log resolution, then $f$ will be referred to as a \emph{crepant resolution}.
\end{definition}

Now let $X$ be a normal variety with at worst Gorenstein canonical singularities, $\rho:Z\to X$ a log resolution of $X$, and let $D_{\rho}=\sum_{i\in I} a_iD_i$ denote the discrepancy divisor $K_Z-\rho^*K_X$. For $J$ a subset of the index set $I$ of the irreducible components of $D_{\rho}$, we let $D_J=\bigcap_{j\in J}D_j$, with the convention that $D_{\varnothing}=Z$. Techniques of motivic integration may then be used to show the class\footnote{The map $\rho_*:A_*Z\to A_*X$ is the proper pushforward associated with the map $\rho:Z\to X$. Moreover, by $\mathfrak{C}(D_J)$ in \eqref{shcd} we really mean $i_{J*}\mathfrak{C}(D_J)$, where $i_J$ denotes the inclusion $i_J:D_J\hookrightarrow Y$. Also, if $\int_{\mathbb{A}^1}\mathfrak{C}(\mathbb{A}^1)=1$, replace $\prod_{j\in J}\frac{\int_{\mathbb{A}^1}\mathfrak{C}(\mathbb{A}^1)-\left(\int_{\mathbb{A}^{1}}\mathfrak{C}(\mathbb{A}^{1})\right)^{a_j+1}}{\left(\int_{\mathbb{A}^{1}}\mathfrak{C}(\mathbb{A}^{1})\right)^{a_j+1}-1}$ by $\prod_{j\in J}\frac{-a_j}{a_j+1}=\lim_{t\to 1}\frac{t-t^{a_j+1}}{t^{a_j+1}-1}$ in \eqref{shcd}.}
\begin{equation}\label{shcd}
\rho_*\left(\sum_{J\subset I}\mathfrak{C}(D_J)\cdot \prod_{j\in J}\frac{\int_{\mathbb{A}^1}\mathfrak{C}(\mathbb{A}^1)-\left(\int_{\mathbb{A}^{1}}\mathfrak{C}(\mathbb{A}^{1})\right)^{a_j+1}}{\left(\int_{\mathbb{A}^{1}}\mathfrak{C}(\mathbb{A}^{1})\right)^{a_j+1}-1}\right)
\end{equation}
is independent of the resolution $\rho:Z\to X$ (\cite{DLMI}, Proposition~6.3.2 ), and as such, it is an intrinsic invariant of $X$, which leads us to the following
\begin{definition}\label{D3}
The class \eqref{shcd} will be referred to as the \emph{stringy} $\mathfrak{C}$-class of $X$, and will be denoted $\mathfrak{C}^{\text{str}}(X)$. 
\end{definition}
It immediately follows from Definition~\ref{D3} that if the map $\rho$ appearing in \eqref{shcd} is crepant, then $\mathfrak{C}^{\text{str}}(X)=\rho_*\mathfrak{C}(Z)$, and in particular, if $X$ is smooth then $\mathfrak{C}^{\text{str}}(X)=\mathfrak{C}(X)$. 
 
Now suppose $V$ and $W$ are $K$-equivalent varieties with at worst Gorenstein canonical singularities, so that there exists a common log resolution 
\[  
\xymatrix{
 &Z \ar[dl]_{v} \ar[dr]^{w}  & \\
V& & W \\
}
\]
such that $v^*K_V=w^*K_W$. It then follows that the the discrepancy divisors $D_v=K_Z-v^*K_V$ and $D_w=K_Z-w^*K_W$ coincide, thus by Definition~\ref{D3} there exists a class $C\in A_*Z\otimes \text{Frac}(R)$ such that
\[
v_*(C)=\mathfrak{C}^{\text{str}}(V), \quad \text{and} \quad w_*(C)=\mathfrak{C}^{\text{str}}(W).
\] 
This further implies
\begin{equation}\label{e1}
\int_V\mathfrak{C}^{\text{str}}(V)=\int_W\mathfrak{C}^{\text{str}}(W), 
\end{equation}
and moreover, if $f:X\to Y$ is crepant then
\begin{equation}\label{e2}
\int_X\mathfrak{C}^{\text{str}}(X)=\int_Y\mathfrak{C}^{\text{str}}(Y).
\end{equation}

With regards towards the cases where $\mathfrak{C}$ is either the Hirzebruch class $\widetilde{T}_y$ or the normalized Hirzebruch class $T_y$ as defined in \S\ref{intro}, equation \eqref{d0} along with the definition of $\mathfrak{C}^{\text{str}}$ implies
\begin{equation}\label{scyc}
\int_X \widetilde{T}_y^{\text{str}}(X)=\int_X T_y^{\text{str}}(X),
\end{equation}
and as such, we take either side of equation \eqref{scyc} as the definition of the \emph{stringy} $\chi_y$-\emph{characteristic} of $X$, which we denote by $\chi_y^{\text{str}}(X)$. The stringy $\chi_y$-characteristic $\chi_y^{\text{str}}(X)$ evaluated at $y=-1$ then yields the stringy Euler characteristic $\chi_{\text{str}}(X)$ as defined by Batyrev in \cite{SHNB}. We note that while Hirzebruch's $\chi_y$-genus is a polynomial in $y$, $\chi_y^{\text{str}}(X)$ is a priori a rational function in $y$. In the case that $X$ admits a crepant resolution $\rho:Z\to X$, the coefficient of $y^p$ in $\chi_y^{\text{str}}(X)$ encodes linear relations between the stringy Hodge numbers $h_{\text{str}}^{p,q}(X)$ for $q=0,...,\text{dim}(X)$ (which in such a case coincide with the usual Hodge numbers $h^{p,q}(Z)$), thus yielding a `stringy' version of the Hirzebruch-Riemann-Roch theorem. For more on stringy Hodge numbers, see \cite{SHNB}.

\section{Computing stringy characteristic classes}\label{CSCC}
Let $\mathfrak{C}$ be a cohomological characteristic class for vector bundles which is multiplicative on short exact sequences, so that if
\[
0\longrightarrow \mathscr{A} \longrightarrow \mathscr{B} \longrightarrow \mathscr{C} \longrightarrow 0
\]
is a short exact sequence of vector bundles, then
\[
\mathfrak{C}(\mathscr{B})=\mathfrak{C}(\mathscr{A})\mathfrak{C}(\mathscr{C}). 
\]
If $Y$ is a smooth variety, we recall $\mathfrak{C}^{\text{str}}(Y)=\mathfrak{C}(Y)=\mathfrak{C}(TY)\cap [Y]$. We note that we don't make the assumption that $\mathfrak{C}(\mathscr{O})=1$, as to incorporate the (non-normalized) Hirzebruch class $\widetilde{T}_y$. For $Y$ with at worst Gorenstein canonical singularities, assume $Y$ is a complete intersection in a smooth ambient variety $Z$. We then consider crepant resolutions of $Y$ which are obtained by successively blowing up $Z$ along (possibly singular) complete intersections and then taking a proper transform of $Y$ along the blowups. We now outline a general method for computing $\mathfrak{C}^{\text{str}}(Y)$ via such a resolution. 

A complete intersection is a special case of a regular embedding, whose definition we now recall. A closed embedding $Y\hookrightarrow Z$ is said to be \emph{regular} of codimension $d$ if the ideal sheaf $\mathscr{I}_Y$ of $Y$ in $Z$ is such that $\mathscr{I}_Y(U)$ locally generated by a regular sequence in $\mathscr{O}_Z(U)$ of length $d$ for every affine open set $U\subset Z$. In such a case, the normal cone to $Y$ in $Z$ is in fact a vector bundle $N_YZ\to Y$ of rank $d$, which we refer to as the \emph{normal bundle} to $Y$ in $Z$. The sheaf of sections of $N_YZ\to Y$ is then dual to $\mathscr{I}_Y/\mathscr{I}_Y^2$. A regular embedding $Y\hookrightarrow Z$ of codimension $d$ is said to be a \emph{complete intersection} if there exists a vector bundle $\mathcal{N}\to Z$ of rank $d$ such that the image of the embedding $Y\hookrightarrow Z$ may be identified with the zero-scheme of a regular section of $\mathcal{N}\to Z$. In such a case, the vector bundle $\mathcal{N}\to Z$ restricts to $Y$ as $N_YZ$. We now recall some facts about blowups of regular embeddings along a common regularly embedded subscheme.
\begin{lemma}\label{L1}
Let $X\longrightarrow Y\longrightarrow Z$ be a sequence of regular embeddings with $Z$ smooth, and let $p:\widetilde{Z}\to Z$ and $q:\widetilde{Y}\to Y$ be the blowups of $Z$ and $Y$ along $X$ respectively. Then 
\\

(i) the composition $X\longrightarrow Z$ is a regular embedding, and there exists a short exact sequence of vector bundles $$0\longrightarrow N_{X}Y \longrightarrow N_{X}Z \longrightarrow \left.N_{Y}Z\right|_{X} \longrightarrow 0.$$

(ii) $X$ may be realized as the zero-scheme of a regular section of a vector bundle $\mathscr{E}\to Z$. The normal bundle to $X$ in $Z$, denoted $N_XZ$, then embeds as a subbundle of $\left.\mathscr{E}\right|_X$. If $X$ is in fact a complete intersection in $Z$, then $\left.\mathscr{E}\right|_X=N_XZ$.\\

(iii) there exists a sequence of regular embeddings $\widetilde{Y} \overset{i}\longrightarrow \widetilde{Z}\overset{j}\longrightarrow \mathbb{P}(\mathscr{E})$ which forms the top row of the following commutative diagram
\begin{equation}\label{bd}
\xymatrix{
&\widetilde{Y} \ar[d]^q \ar[r]^i &\widetilde{Z} \ar[d]^p \ar[r]^j & \mathbb{P}(\mathscr{E}) \ar[dl]^{\pi} \\
X\ar[r] &Y \ar[r] & Z,
}
\end{equation}
where $\pi:\mathbb{P}(\mathscr{E})\to Z$ denotes the projective bundle of lines in $\mathscr{E}$. In particular, $\widetilde{Y}$ and $\widetilde{Z}$ are both regularly embedded in the smooth variety $\mathbb{P}(\mathscr{E})$.
\\

(iv) $\left.\mathscr{O}(-1)\right|_{\widetilde{Z}}=\mathscr{O}(E_p)$ and $\left.\mathscr{O}(-1)\right|_{\widetilde{Y}}=\mathscr{O}(E_q)$, where $E_p$ and $E_q$ are the exceptional divisors of the blowups $p:\widetilde{Z}\to Z$ and $q:\widetilde{Y}\to Y$ respectively, and $\mathscr{O}(-1)$ is the tautological line bundle of $\pi:\mathbb{P}(\mathscr{E})\to Z$.\\

(v) $\widetilde{Y}$ is the proper transform of $Y$ under the blowup $p:\widetilde{Z}\to Z$, and  $$N_{\widetilde{Y}}\widetilde{Z}=q^*N_YZ\otimes \mathscr{O}(-E_q).$$
\end{lemma}
\begin{proof}
$(i)$ follows from B.7.4 in \cite{IT}. $(ii)$ and $(iii)$ follow from B.8.2 in \cite{IT}. $(iv)$ follows B.6.3 and B.6.9 in \cite{IT}. $(v)$ follows from B.6.10 in \cite{IT}. 
\end{proof}
We now arrive at the following
\begin{proposition}\label{P1}
Under the assumptions of Lemma~\ref{L1}, if $\widetilde{Y}$ is smooth then
\begin{equation}\label{gccf}
\mathfrak{C}(T\widetilde{Y})=\frac{\mathfrak{C}(\mathscr{O}(E_q))\mathfrak{C}(\left.p^*\mathscr{E}\right|_{\widetilde{Y}}\otimes \mathscr{O}(-E_q))\mathfrak{C}(\left.p^*TZ\right|_{\widetilde{Y}})}{\mathfrak{C}(\mathscr{O}_{\widetilde{Y}})\mathfrak{C}(q^*N_YZ\otimes \mathscr{O}(-E_q))\mathfrak{C}(\left.p^*\mathcal{N}\right|_{\widetilde{Y}})\mathfrak{C}(\left.p^*\mathcal{Q}\right|_{\widetilde{Y}}\otimes \mathscr{O}(-E_q))}
\end{equation}
where $\mathcal{N}$ and $\mathcal{Q}$ are bundles on $Z$ which restrict to $X$ as $N_XZ$ and the quotient bundle $\left.\mathscr{E}\right|_{X}/N_XZ$, respectively \footnote{The RHS of  formula \eqref{gccf} still makes sense when the bundle $\mathcal{N}$ doesn't exist, since one can first cap classes with $\mathfrak{C}(\mathscr{O}(E_q))$ which yields classes supported on subvarieties of $E_q$, and then instead of capping with $\mathfrak{C}(p^*\mathcal{N})\mathfrak{C}(p^*\mathcal{Q}\otimes \mathscr{O}(-E_q))$ one can then cap with with $\mathfrak{C}(r^*N_XZ)\mathfrak{C}(r^*Q\otimes \mathscr{O}(-\left.E_q\right|_{E_q}))$, where $r=\left.q\right|_{E_q}:E_q\to X$ and $Q=\left.\mathscr{E}\right|_{X}/N_XZ$.}. In particular, if $X$ is in fact a complete intersection in $Z$, so that $N_XZ=\left.\mathscr{E}\right|_{X}=\left.\mathcal{N}\right|_{X}$ and $\mathcal{Q}=0$, and $\mathfrak{C}(\mathscr{O})=1$, then
\begin{equation}\label{sccf}
\mathfrak{C}(T\widetilde{Y})=\frac{\mathfrak{C}(\mathscr{O}(E_q))\mathfrak{C}(\left.p^*\mathcal{N}\right|_{\widetilde{Y}}\otimes \mathscr{O}(-E_q))\mathfrak{C}(\left.p^*TZ\right|_{\widetilde{Y}})}{\mathfrak{C}(q^*N_YZ\otimes \mathscr{O}(-E_q))\mathfrak{C}(\left.p^*\mathcal{N}\right|_{\widetilde{Y}})}.
\end{equation} 
\end{proposition}
\begin{proof}
Since $\mathfrak{C}$ is multiplicative it satisfies the adjunction formula, thus
\[
\mathfrak{C}(T\widetilde{Y})=\frac{\mathfrak{C}(\left.T\mathbb{P}(\mathscr{E})\right|_{\widetilde{Y}})}{\mathfrak{C}(N_{\widetilde{Y}}\mathbb{P}(\mathscr{E}))}.
\]
By B.5.8 in \cite{IT} we have
\[
\mathfrak{C}(T\mathbb{P}(\mathscr{E}))=\frac{\mathfrak{C}(\pi^*\mathscr{E}\otimes \mathscr{O}(1))\pi^*\mathfrak{C}(TZ)}{\mathfrak{C}(\mathscr{O}_{\mathbb{P}(\mathscr{E})})},
\]
and since by Lemma~\ref{L1} $\left.\mathscr{O}(1)\right|_{\widetilde{Y}}=\mathscr{O}(-E_q)$, restricting to $\widetilde{Y}$ yields
\[
\mathfrak{C}(\left.T\mathbb{P}(\mathscr{E})\right|_{\widetilde{Y}})=\frac{\mathfrak{C}(\left.p^*\mathscr{E}\right|_{\widetilde{Y}}\otimes \mathscr{O}(-E_q))\mathfrak{C}(\left.p^*TZ\right|_{\widetilde{Y}})}{\mathfrak{C}(\mathscr{O}_{\widetilde{Y}})}.
\]
We now compute $\mathfrak{C}(N_{\widetilde{Y}}\mathbb{P}(\mathscr{E}))$. For this, note by Lemma~\ref{L1} there exists a short exact sequence of vector bundles
\[
0\longrightarrow N_{\widetilde{Y}}\widetilde{Z}\longrightarrow N_{\widetilde{Y}}\mathbb{P}(\mathscr{E})\longrightarrow \left.N_{\widetilde{Z}}\mathbb{P}(\mathscr{E})\right|_{\widetilde{Y}}\longrightarrow 0,
\]
so that 
\[
\mathfrak{C}(N_{\widetilde{Y}}\mathbb{P}(\mathscr{E}))=\mathfrak{C}(N_{\widetilde{Y}}\widetilde{Z})\mathfrak{C}(\left.N_{\widetilde{Z}}\mathbb{P}(\mathscr{E})\right|_{\widetilde{Y}}).
\]
Now since $\mathfrak{C}(N_{\widetilde{Y}}\widetilde{Z})=\mathfrak{C}(q^*N_YZ\otimes \mathscr{O}(-E_q))$ by Lemma~\ref{L1}, the proposition is proved once we show 
\begin{equation}\label{nbf}
\mathfrak{C}(\left.N_{\widetilde{Z}}\mathbb{P}(\mathscr{E})\right|_{\widetilde{Y}})=\frac{\mathfrak{C}(\left.p^*\mathcal{N}\right|_{\widetilde{Y}})\mathfrak{C}(\left.p^*\mathcal{Q}\right|_{\widetilde{Y}}\otimes \mathscr{O}(-E_q))}{\mathfrak{C}(\mathscr{O}(E_q))},
\end{equation}
where $\mathcal{N}$ and $\mathcal{Q}$ are bundles on $Z$ which restrict to $N_XZ$ and the quotient bundle $\left.\mathscr{E}\right|_{X}/N_XZ$ respectively (if such a bundle $\mathcal{N}$ does not exist, see the footnote referenced in the statement of Proposition~\ref{P1}). For $\mathfrak{C}$ the Chern class, formula \eqref{nbf} follows from Theorem~1.2 in \cite{ACCB}, which in the context at hand says
\begin{equation}\label{ccbf}
c(N_{\widetilde{Z}}\mathbb{P}(\mathscr{E}))=\frac{c(p^*\mathcal{N})c(p^*\mathcal{Q}\otimes \mathscr{O}(-E_p))}{c(\mathscr{O}(E_p))}.
\end{equation}
But going through the proof of Theorem~1.2 in \cite{ACCB}, one sees that formula \eqref{ccbf} is essentially obtained by equating Chern classes of normal bundles via exact sequences obtained from the following fiber square of regular embeddings
\begin{equation}\label{refs}
\xymatrix{
\widetilde{Z} \ar[d] \ar[r] & \mathbb{P}(\mathscr{E}) \ar[d] \\
\widetilde{\mathscr{E}} \ar[r] &\mathbb{P}(f^*\mathscr{E}),
}
\end{equation}
where $\widetilde{\mathscr{E}}\to \mathscr{E}$ is the blowup of $\mathscr{E}$ along the zero section of $f:\mathscr{E}\to Z$. In particular, diagram \eqref{refs} along with Lemma~\ref{L1} yields
\begin{equation}\label{abf}
c(N_{\widetilde{Z}}\mathbb{P}(\mathscr{E}))=\frac{c(\left.N_{\widetilde{\mathscr{E}}}\mathbb{P}(f^*\mathscr{E})\right|_{\widetilde{Z}})c(N_{\widetilde{Z}}\widetilde{\mathscr{E}})}{c(\left.N_{\mathbb{P}(\mathscr{E})}\mathbb{P}(f^*\mathscr{E})\right|_{\widetilde{Z}})}.
\end{equation}
Then using the fact that 
\[
\left.N_{\mathbb{P}(\mathscr{E})}\mathbb{P}(f^*\mathscr{E})\right|_{\widetilde{Z}}\cong p^*\mathscr{E} \quad \text{and} \quad \left.N_{\widetilde{\mathscr{E}}}\mathbb{P}(f^*\mathscr{E})\right|_{\widetilde{Z}}\cong p^*\mathscr{E}/\mathscr{O}(E_p),
\]
equation \eqref{abf} then simplifies to 
\[
c(N_{\widetilde{Z}}\mathbb{P}(\mathscr{E}))=\frac{c(N_{\widetilde{Z}}\widetilde{\mathscr{E}})}{c(\mathscr{O}(E_p))}.
\]
The proof then concludes by showing $c(N_{\widetilde{Z}}\widetilde{\mathscr{E}})=c(p^*\mathcal{N})c(p^*\mathcal{Q}\otimes \mathscr{O}(-E_p))$, which again uses the technique of constructing sequences of regular embeddings and the taking Chern classes of the associated short exact sequence of normal bundles. As such, at each stage of the proof $c$ may be replaced by $\mathfrak{C}$, and after restriction to $\widetilde{Y}$ we then arrive at equation \eqref{nbf}, as desired.
\end{proof}

Now suppose $Y_0\hookrightarrow Z_0$ is a complete intersection with $Z_0$ a smooth variety, and assume $Y_0$ has at worst Gorenstein canonical singularities. We then assume there exists a sequence $n$ blowups
\[
Z_n\to Z_{n-1}\to \cdots \to Z_1\to Z_0, 
\]
such that the proper transform of $Y_0$ through the blowups yields a crepant resolution
\[
\rho:Y_n\longrightarrow Y_0.
\] 
We further assume that $Z_{i+1}\to Z_{i}$ is the blowup of $Z_{i}$ along a (possibly singular) complete intersection $X_{i}\to Z_i$, and we assume that $X_i$ is a complete intersection in the proper transform $Y_i$ of $Y_{i-1}$, so that for $i=0,...,n-1$ we have a sequence of regular embeddings $X_i\to Y_i\to Z_i$. By Lemma~\ref{L1} we then have the following commutative diagram
\begin{equation}\label{bd2}
\xymatrix{
                    & Y_n \ar[r]^{i_n} \ar[d]^{q_n} & Z_n \ar[d]^{p_n}  &    \\
X_{n-1} \ar[r] & Y_{n-1} \ar[r]^{i_{n-1}} \ar[d]^{q_{n-1}} & Z_{n-1} \ar[d]^{p_{n-1}}  &      \\
                    & \vdots   \ar[d]^{q_2}              & \vdots \ar[d]^{p_2}  &               \\
X_1 \ar[r]       &Y_1 \ar[r]^{i_1} \ar[d]^{q_1}  &Z_1 \ar[d]^{p_1}\ar[r]^{j_1} & \mathbb{P}(\mathscr{E}_0) \ar[dl]^{\pi_1} \\
X_0      \ar[r]  & Y_0 \ar[r]                    & Z_0,        &        \\
}
\end{equation}
where $\mathscr{E}_0$ is a vector bundle on $Z_0$ such that $X_0$ may be realized as the zero-scheme of a regular section of $\mathscr{E}_0\to Z_0$, and $p_i$ and $q_i$ are the blowups of $Z_{i-1}$ and $Y_{i-1}$ along $X_{i-1}$ respectively. The crepant resolution $\rho:Y_n\to Y_0$ is then given by $\rho=q_1\circ \cdots \circ q_n$. We now wish to compute $\mathfrak{C}(Y_n)$ via an adjunction formula, which exists if $Y_n$ is embedded in a smooth variety. An explicit embedding of $Y_n$ in a smooth variety then follows from the following 
\begin{claim}\label{C1}
Let $Z_i$ be as in diagram \eqref{bd2} for $i=1,...,n$. Then there exists a regular embedding $Z_i\longrightarrow W_i$ with $W_i$ smooth. In particular, $W_i$ may be defined inductively as  $W_i=\mathbb{P}(\mathscr{E}_{i-1})$, where $\mathscr{E}_{i-1}\to W_{n-1}$ is a vector bundle which admits a regular section whose zero-scheme is $X_{i-1}$.  
\end{claim} 
\begin{proof}
We use induction on $i$. For $i=1$ we take $j_1:Z_1\longrightarrow \mathbb{P}(\mathscr{E}_0)=W_1$, which is a regular embedding of $Z_1$ into a smooth variety by Lemma~\ref{L1}. Now assume the result holds for $i=k$ with $1\leq k<n$. Then again by Lemma~\ref{L1} we have the following diagram
\[
\xymatrix{
&Y_{k+1}\ar[r] \ar[d] &Z_{k+1}\ar[r] \ar[d]&\widetilde{W_k}\ar[r] \ar[d]& \mathbb{P}(\mathscr{E}_k)\ar[dl] \\
X_k\ar[r]&Y_k\ar[r] &Z_k\ar[r] &W_k,
}
\]
where the vertical arrows are all blowups along $X_k$, and $\mathscr{E}_k\to W_k$ is a vector bundle which admits a regular section whose zero-scheme is $X_k$. We then take the composition $Z_{k+1}\longrightarrow \mathbb{P}(\mathscr{E}_k)=W_{k+1}$ in the above diagram, which completes the proof.
\end{proof}

We note that if $X_i$ is smooth for $i=1,...,n-1$, $Z_i$ will be smooth for $i=1,...,n$, thus in such a case the embedding $i_n:Y_n\to Z_n$ is sufficient for adjunction. Claim~\ref{C1} then becomes relevant only if $X_i$ is singular for some $i$. In any case, by Claim~\ref{C1}, we may supplement the top of diagram \eqref{bd2} as follows
\begin{equation}\label{bd3}
\xymatrix{
&Y_{n}\ar[r]^{i_n} \ar[d]^{q_n} &Z_{n}\ar[r]^{j_n} \ar[d]^{p_n}&\widetilde{W_{n-1}}\ar[r]^{k_n} \ar[d]^{r_n}& \mathbb{P}(\mathscr{E}_{n-1})\ar[dl]^{\pi_n} \\
X_{n-1}\ar[r]&Y_{n-1}\ar[r] &Z_{n-1}\ar[r] &W_{n-1}.
}
\end{equation}
By Proposition~\ref{P1} we then arrive at the following
\begin{proposition}\label{P2}
Under the assumptions and notation of diagram \eqref{bd3}, we have
\[
i_{n*}\mathfrak{C}(Y_n)=\frac{\mathfrak{C}(\mathscr{O}(E_n))\mathfrak{C}(\mathcal{N}\otimes \mathscr{O}(-E_n))p_n^*\mathfrak{C}(W)}{\mathfrak{C}(\mathscr{O}_{Z_n})\mathfrak{C}(N_{Y_n})\mathfrak{C}(p_n^*N_{Z_{n-1}}\otimes \mathscr{O}(-E_n))\mathfrak{C}(\mathcal{N})}\cap [Y_n].
\]
where $W=\left.TW_{n-1}\right|_{Z_{n-1}}$, $N_{Y_n}=N_{Y_n}Z_n$, $N_{Z_{n-1}}=N_{Z_{n-1}}W_{n-1}$, and $\mathcal{N}\to Z_{n}$ is a restriction to $Z_n$ of a bundle on $W_{n-1}$ which restricts to $X_{n-1}$ as $N_{X_{n-1}}W_{n-1}$.
\end{proposition}

We now recall $\iota:Y_0\to Z_0$ denotes the inclusion of the singular variety $Y_0$ in the smooth ambient space $Z_0$. Now since $\rho:Y_n\to Y_0$ is a crepant resolution, we then have 
\[
\iota_*\mathfrak{C}^{\text{str}}(Y_0)=\tau_*(i_{n*}\mathfrak{C}(Y_n)),
\]
where $\tau=p_1\circ \cdots \circ p_n$. As such, to complete the computation of $\iota_*\mathfrak{C}^{\text{str}}(Y_0)$, we now prove a lemma which suffices to effectively compute $\tau_*$ of any class in $A_*Z_n$, which essentially boils down to computing $p_{i*}E_i^k$ for $i=1,...,n$, where $E_i^k$ denotes the $k$-fold intersection product of the exceptional divisor $E_i$ of $p_i:Z_i\to Z_{i-1}$ with itself. 

The formula we now prove for pushing forward powers of the exceptional divisor of a blowup is a direct generalization of a formula appearing in \cite{FvH}. Whereas the formula appearing in \cite{FvH} is for blowups along smooth complete intersections cut out by smooth divisors, we generalize the situation to blowups along possibly singular complete intersections, and no smoothness assumption is made on the divisors which cut out the complete intersection.  

\begin{lemma}\label{L2}
Let $\iota:X\to Z$ be a (possibly singular) complete intersection of codimension $d$ with $Z$ smooth, $p:\widetilde{Z}\to Z$ be the blowup of $Z$ along $X$, and let $E $ denote the exceptional divisor of $p$. Then if $U_1,\ldots, U_d$ are classes of divisors which cut out $X$ in $Z$, then
\begin{equation}\label{ped}
p_*E^k=\sum_{i=1}^d\left(\prod_{j\neq i}\frac{U_j}{U_j-U_i}\right)U_i^k.
\end{equation}
In particular, if $g(E) = \sum_{k=0}^{\infty}  a_k E^k$ is a formal power series in $E$ with coefficients $a_k\in A_*\widetilde{Z}$ and $C_i = \prod_{j \neq i} U_j/(U_j-U_i)$, then
\begin{equation}\label{plugitin}
f_*g(E) =C_1 g(U_1) + \cdots +C_d g(U_d).
\end{equation}
\end{lemma}
\begin{proof}
The Segre class of $E$ in $\widetilde{Z}$, denoted $s(E,\widetilde{Z})$, is by definition given by
\[
s(E,\widetilde{Z})=c(N_E\widetilde{Z})^{-1}\cap [E] \in A_*E,
\]
so its pushforward to $\widetilde{Z}$ is then given by
\[
\tilde{\iota}_*s(E,\widetilde{Z})=c(\mathscr{O}(E))^{-1}\cap [E]=E-E^2+E^3-\cdots,
\]
where $\tilde{\iota}:E\to \widetilde{Z}$ denotes the inclusion. Now consider the following fiber square
\[
\xymatrix{
E \ar[d]_q \ar[r]^{\tilde{\iota}} & \widetilde{Z} \ar[d]^p \\
X \ar[r]^{\iota} & Z.
}
\]
By the birational invariance of Segre classes (namely, Proposition~4.2 (a) in \cite{IT}), 
\[
q_*s(E,\widetilde{Z})=s(X,Z),
\]
thus
\[
p_*\tilde{\iota}_*s(E,\widetilde{Z})=\iota_*q_*s(E,\widetilde{Z})=\iota_*s(X,Z)=c(\mathcal{N})^{-1}\cap [X],
\] 
where $\mathcal{N}\to Z$ is a vector bundle such that $\left.\mathcal{N}\right|_{X}=N_XZ$. Now if $U_1,\ldots, U_d$ denote classes of divisors which cut out $X$ in $Z$, so that $[X]=U_1\cdots U_d\in A_*Z$, then
\[
c(\mathcal{N})^{-1}\cap [X]=\frac{U_1\cdots U_d}{(1+U_1)\cdots (1+U_d)},
\]
where $1/(1+U_i)$ is shorthand for the formal series expansion $1-U_i+U_i^2-\cdots$. Putting things together then yields
\[
p_*(E-E^2+E^3-\cdots)=\frac{U_1\cdots U_d}{(1+U_1)\cdots (1+U_d)},
\]
so that 
\[
p_*E^k=(-1)^{k+1}\cdot [t^k]\frac{t^dU_1\cdots U_d}{(1+tU_1)\cdots (1+tU_d)}
\] 
where if $h(t)=\sum_i b_it^i$ is a formal power series in $t$ then $[t^m]h(t)=b_m$. Equation \eqref{ped} then follows by elementary manipulations of formal power series. Now if $g(E) = \sum_{k=0}^{\infty}  a_k E^k$ is a formal power series in $E$ with coefficients $a_k\in A_*\widetilde{Z}$ and $C_i = \prod_{j \neq i} U_j/(U_j-U_i)$, then
\[
f_*g(E) =\sum_{k=0}^{\infty}a_kf_*(E^k) =  \sum_{k=0}^{\infty}  a_k (C_1 U_1^n + \cdots +C_d U_d^n) = C_1 g(U_1) + \cdots +C_d g(U_d),
\]
thus completing the proof.
\end{proof}

We note that in the case $\iota:X\to Z$ is a smooth complete intersection cut out by smooth divisors, formula \eqref{plugitin} is referred to as `Theorem~1.8' in \cite{EJK}, but in any case, formula \eqref{plugitin} follows immediately from the pushforward formula \eqref{ped}.


\section{Weierstrass fibrations in $F$-theory}\label{WMFT}
$F$-theory is a geometrization of the $\text{SL}_2(\Z)$ symmetry of non-perturbative type-IIB string theory \cite{VFT}. In particular, a type-IIB compactification of $10-2n$ spacetime dimensions on a compact $n$-fold $B$ comes with a fundamental $\text{SL}_2(\Z)$-invariant complex scalar field $\tau$ referred to as the \emph{axio-dilaton}, which is given by
\[
\tau=C_{(0)}+ie^{-\phi},
\]   
where $C_{(0)}$ is the RR-scalar field and $\phi$ is coming from the NS-NS sector, which is related to the string coupling parameter $g_s$ by the equation $e^{\phi}=g_s$. In perturbative type-IIB string theory, $\tau$ is taken to be constant and $B$ is assumed to be Calabi-Yau in order to preserve super-symmetry. In its non-perturbative regime, the axio-dilaton $\tau$ of type-IIB is assumed to be a varying complex scalar field over $B$, and $B$ is then no longer required to be Calabi-Yau as the non-constant behavior of $\tau$ preserves super-symmetry. In $F$-theory, the $\text{SL}_2(\Z)$-invariant, non-constant complex scalar field $\tau$ of non-perturbative type-IIB is then identified with the complex structure parameter of an \emph{actual} family of elliptic curves, varying over the $n$-fold $B$ according to the behavior of the axio-dilaton $\tau$. This geometric viewpoint of non-perturbative type-IIB is then encapsulated in an elliptic fibration $Y\to B$, whose total space $Y$ is a Calabi-Yau $(n+1)$-fold.

The $F$-theoretic geometrization of type-IIB via an elliptic fibration $Y\to B$ has been a useful tool for probing the non-perturbative aspects of type-IIB string compactifications. In particular, using Weierstrass fibrations along with Tate's algorithm, one may geometrically engineer non-Abelian gauge theories \cite{GSEGS}\cite{TAFTK}. For this, one starts with a Weierstrass fibration $Y\to B$ in Tate form, so that the total space $Y$ is a hypersurface in a $\mathbb{P}^2$-bundle $\pi:\mathbb{P}(\mathscr{E})\to B$ given by
\begin{equation}\label{tf}
Y:(y^2z + a_1xyz + a_3yz^2 = x^3 + a_2x^2z + a_4xz^2 + a_6z^3)\subset \mathbb{P}(\mathscr{E}),
\end{equation}
where $\mathscr{E}=\mathscr{O}_B\oplus \mathscr{L}^2\oplus \mathscr{L}^3$ and $\mathscr{L}\to B$ is a line bundle referred to as the \emph{fundamental line bundle} of $Y\to B$. Taking $x$ to be a section of $\mathscr{O}(1)\otimes \pi^*\mathscr{L}^2$, $y$ to be a section of $\mathscr{O}(1)\otimes \pi^*\mathscr{L}^3$, $z$ to be a section of the tautological bundle $\mathscr{O}(1)$ and $a_i$ to be a section of $\pi^*\mathscr{L}^i$ then identifies $Y$ with the zero-scheme of a section of $\mathscr{O}(3)\otimes \pi^*\mathscr{L}^6$. The canonical class $K_Y$ is then trivial if and only if $\mathscr{L}=\mathscr{O}(-K_B)$, so in the context of string compactifications $\mathscr{L}$ is always taken to be the anti-canonical bundle of $B$. The discriminant of $Y$ is then a divisor in $B$ which parametrizes the singular fibers of $Y$, which is given by
\[
\Delta_Y:(4F^3+27G^2=0)\subset B,
\]
where 
\[
\begin{cases}
F=\frac{-1}{48}(b_2^2-24b_4) \\
G=\frac{-1}{864}(36b_2b_4 - b_2^3 - 216b_6) \\
b_2 = a_1^2 + 4a_2 \\
b_4 = a_1a_3 + 2a_4 \\
b_6 = a_3^2 + 4a_6.
\end{cases}
\]
Moreover, the Weierstrass equation for $Y$ is then given by 
\[
Y:(y^2z=x^3+Fxz^2+Gz^3)\subset \mathbb{P}(\mathscr{E}).
\]

A gauge group $\mathcal{G}_Y$ is then associated with $Y$ as follows (see \cite{EJK} for further details). By Tate's algorithm \cite{Tate}, one may perturb the coefficient sections $a_i$ in such a way that a particular reducible singular fiber $\mathfrak{f}_Y$ appears generically over an irreducible component $S$ of the discriminant $\Delta_Y$ after a resolution of singularities. In such as case, the equation of $\Delta_Y$ necessarily takes the form $s^kt=0$, where $s$ is a regular section of $\mathscr{O}(S)$, and the singular locus of $Y$ is given by $\{x=y=s=0\}\subset \mathbb{P}(\mathscr{E})$. The dual graph of $\mathfrak{f}_Y$ is then an affine Dynkin diagram associated with a Lie algebra $\mathfrak{g}$. The gauge group $\mathcal{G}_Y$ is then given by
\[
\mathcal{G}_Y=\frac{\text{exp}(\mathfrak{g}^{\vee})}{\text{MW}_{\text{tor}}(\varphi)}\times U(1)^{\text{rkMW}(\varphi)},
\]  
where $\varphi:Y\to B$ denotes the projection associated with the elliptic structure of $Y$ as given by \eqref{tf}, and $\text{MW}(\varphi)$ denotes the Mordell-Weil group of rational sections of $\varphi$ (whose torsion subgroup is denoted $\text{MW}_{\text{tor}}(\varphi)$). We note that it is possible for two distinct Weierstrass fibrations $Y\to B$ and $Y'\to B$ with distinct $\mathfrak{f}_Y$ and $\mathfrak{f}_{Y'}$ to give rise to the same gauge group, so that it is not necessarily the case that $\mathcal{G}_Y\neq \mathcal{G}_{Y'}$. 

In this note, we consider 14 different families of Weierstrass fibrations $Y\to B$ for which $\mathcal{G}_Y$ simple. In \cite{EJK}, crepant resolutions of the Weierstrass fibrations we consider were constructed over an unspecified base $B$ of arbitrary dimension, and moreover, the resolutions do not require that the fundamental line bundle $\mathscr{L}$ is the anti-canonical line bundle of $B$ (so that the Calabi-Yau case is recovered by setting $\mathscr{L}=\mathscr{O}(-K_B)$). We now give the equations for such Weierstrass fibrations, along with the details of the crepant resolutions constructed in \cite{EJK}.

\begin{table}[hbt]\label{t1}
\begin{center}
\begin{tabular}{|c|c|c|c|}
\hline
$Y$ & $\mathcal{G}_Y$ \\
\hline
$y^2z=x^3 + a_{4,1}sxz^2 + a_{6,2}s^2z^3$ & SU(2) \\
\hline
$y^2z + a_{3,1}syz^2=x^3 + a_{4,2}s^2xz^2 + a_{6,3}s^3z^3$ & SU(3) \\
\hline
$y^2z + a_1xyz=x^3 + a_{2,1}sx^2z + a_{4,2}s^2xz^2 + a_{6,4}s^{4}z^3$ & SU(4) \\
\hline
$y^2z + a_1xyz + a_{3,2}s^2yz^2=x^3 + a_{2,1}sx^2z + a_{4,3}s^{3}xz^2 + a_{6,5}s^{5}z^3$ & SU(5) \\
\hline
$y^2z=x^3 + a_2x^2z + a_{4,3}s^{3}xz^2 + a_{6,5}s^{5}z^3$ & USp(4) \\
\hline
 $y^2z=x(x^2 + a_2xz + sxz^2)$ & SO(3) \\
\hline
$y^2z=(x^3 + a_2x^2z + s^2xz^2)$ & SO(5) \\
\hline
 $y^2z + a_1xyz=x^3 + sx^2z + s^2xz^2$ & SO(6) \\
\hline
 $y^2z=x^3 + a_{2,1}sx^2z + a_{4,2}s^2xz^2 + a_{6,4}s^4z^3$ & Spin(7) \\
\hline
$y^2z=x^3 + a_{4,2}s^2xz^2 + a_{6,3}s^3z^3$ & $G_2$ \\
\hline
$y^2z=x^3 + a_{4,3}s^3xz^2 + a_{6,4}s^4z^3$ & $F_4$ \\
\hline
$y^2z + a_{3,2}s^2yz^2=x^3 + a_{4,3}s^3xz^2 + a_{6,5}s^5z^3$ & $E_6$ \\
\hline
$y^2z=x^3 + a_{4,3}s^3xz^2 + a_{6,5}s^5z^3$ & $E_7$ \\
\hline
$y^2z=x^3 + a_{4,4}s^4xz^2 + a_{6,5}s^5z^3$ & $E_8$ \\
\hline
\end{tabular}
\end{center}
\caption{Equations for the Weierstrass fibrations we consider along with the associated gauge groups.}
\end{table}

For each Weierstrass fibration listed in Table~1, a crepant resolution was constructed in \cite{EJK} by blowing up the projective bundle $\mathbb{P}(\mathscr{E})$ along smooth complete intersections and then taking the proper transform of $Y$ along the blowups. For example in the SU(3), $G_2$ and USp(4) cases, a crepant resolution is obtained by two blowups. The first blowup $Z_1\to Z_0=\mathbb{P}(\mathscr{E})$ is along its singular locus $\{x=y=s=0\}\subset Z_0$ with exceptional divisor $E_1$, and the second blowup is along $\{y=e_1\}\subset Z_1$, where $e_1$ is a section of $\mathscr{O}(E_1)$ and $y$ denotes the pullback of the section $y$ under $Z_1\to Z_0$. We then summarize the resolution procedure with the notation 
\[
((x,y,s),(y,e_1)),
\]
where the first entry $(x,y,s)$ denotes the ideal along which the first blowup takes place, and the the second entry $(y,e_1)$ denotes the ideal along which the second blowup takes place. Such notation will then be used to summarize each resolution we consider, so that if $Z_i\to Z_{i-1}$ denotes the $i$-th blowup with exceptional divisor $E_i$, then $e_i$ denotes a section of $\mathscr{O}(E_i)$ (we also elide the difference in notation between a section of a line bundle and its pullback via $Z_i\to Z_{i-1}$). 

\begin{table}[hbt]\label{T2}
\begin{center}
\begin{tabular}{|c|c|c|c|}
\hline
Resolution & $\mathcal{G}_Y$ \\
\hline
$(x,y,s)$ & SU(2) \\
\hline
$((x,y,s),(y,e_1))$ & SU(3) \\
\hline
$((x,y,s),(y,e_1),(x,e_2))$ & SU(4) \\
\hline
$((x,y,s),(x,y,e_1),(y,e_1),(y,e_2))$ & SU(5) \\
\hline
$((x,y,s),(y,e_1))$ & USp(4) \\
\hline
 $(x,y)$ & SO(3) \\
\hline
$((x,y,s),(x,y,e_1))$ & SO(5) \\
\hline
 $((x,y,s),(y,e_1),(x,e_2))$ & SO(6) \\
\hline
$((x,y,s),(y,e_1),(x,e_2))$ & Spin(7) \\
\hline
$((x,y,s),(y,e_1))$ & $G_2$ \\
\hline
$((x,y,s),(y,e_1),(x,e_2),(e_2,e_3))$ & $F_4$ \\
\hline
$((x,y,s),(y,e_1),(x,e_2),(e_2,e_3),(y,e_3),(y,e_4))$ & $E_6$ \\
\hline
$((x,y,s),(y,e_1),(x,e_2),(y,e_3),(e_2,e_3),(e_2,e_4),(e_4,e_5))$ & $E_7$ \\
\hline
$((x,y,s),(y,e_1),(x,e_2),(y,e_3),(e_2,e_3),(e_4,e_5),(e_2,e_4,e_6),(e_4,e_7)) $ & $E_8$ \\
\hline
\end{tabular}
\end{center}
\caption{Resolution procedure for each Weierstrass model.}
\end{table}

\section{Stringy Hirzebruch class computation for SO(5) fibrations}\label{SO(5)}
Let $B$ be a smooth compact variety of arbitrary dimension, and let $\varphi:Y\to B$ be the Weierstrass fibration in Table~1 with $\mathcal{G}_Y=\text{SO(5)}$. We now explicitly compute $T_y^{str}(Y)$, which will serve as a representative case for all Weierstrass fibrations in Table~1.

We now set $Y_0=Y$ and $Z_0=\mathbb{P}(\mathscr{O}_B\oplus \mathscr{L}^2\oplus \mathscr{L}^3)$, where $\mathscr{L}\to B$ is the fundamental line bundle. In the notation of \S\ref{CSCC}, we have the following diagram
\begin{equation}\label{bd5}
\xymatrix{
                    & Y_2   \ar[d]_{q_2} \ar[r]^{i_2}  & Z_2 \ar[d]^{p_2}              \\
X_1 \ar[r]       &Y_1 \ar[r]^{i_1} \ar[d]_{q_1}  &Z_1 \ar[d]^{p_1}  \\
X_0      \ar[r]  & Y_0 \ar[r]^{i_0} \ar[d]_{\varphi}     & Z_0 \ar[dl]^{\pi}                \\
                    &B & \\
}
\end{equation}
so that $p_i:Z_i\to Z_{i-1}$ is the blowup of $Z_{i-1}$ along $X_{i-1}$ for $i=1,2$. By Table~2 we have $X_0$ and $X_1$ are smooth complete intersections given by
\[
X_0:(x=y=s=0)\subset Z_0 \quad \quad X_1:(x=y=e_1=0)\subset Z_1.
\]
Since $X_0$ and $X_1$ are both smooth, $Z_1$ and $Z_2$ are both smooth as well, so that the crepant resolution $Y_2\to Y_0$ is a smooth hypersurface in $Z_2$. In particular, in the notation of Claim~\ref{C1} and Proposition~\ref{P2}, we have $W_1=Z_1$ and $W_2=Z_2$, so that both $N_{Z_1}W_1$ and $N_{Z_2}W_2$ are both rank zero vector bundles. Now let $Q(z)$ be the characteristic power series for $T_y$ given by \eqref{npsd}, and let $E_i$ denote the class of the exceptional divisor of $p_i:Z_i\to Z_{i-1}$. Then the divisor class of $Y_2$ in $Z_2$ is given by $[Y_2]=3H+6L-2E_1-2E_2$, where $H$ is the divisor associated with $\mathscr{O}(1)$ on $Z_0$, and $L=c_1(\mathscr{L})$. By Proposition~\ref{P2} we then have
\[
i_{2*}T_y(Y_2)=\mathfrak{q}_2(E_2)p_2^*T_y(Z_1),
\] 
where
\[
\mathfrak{q}_2(E_2)=\frac{Q(E_2)Q(H+2L-E_2)Q(H+3L-E_2)Q(E_1-E_2)\cdot[Y_2]}{Q(3H+6L-2E_1-2E_2)Q(H+2L)Q(H+3L)Q(E_1)},
\]
where we recall the classes of $\{x=0\}$ and $\{y=0\}$ are given by $H+2L$ and $H+3L$ respectively, and $S$ denotes the class of $\{s=0\}$. By the projection formula we then have
\begin{equation}  
p_{2*}\left(i_{2*}T_y(Y_2)\right)=p_{2*}(\mathfrak{q}_2(E_2))T_y(Z_1).
\end{equation}
By Lemma~\ref{L2} we have
\begin{equation}\label{pf2}
p_{2*}(\mathfrak{q}_2(E_2))=\sum_{n=1}^3C_n\mathfrak{q}_2(U_n),
\end{equation}
where $U_1=H+2L$, $U_2=H+3L$, $U_3=E_1$, and $C_n=\prod_{m \neq n} U_m/(U_m-U_n)$. Proposition~\ref{P1} then yields
\[
T_y(Z_1)=\frac{Q(E_1)Q(H+2L-E_1)Q(H+3L-E_1)Q(S-E_1)}{Q(H+2L)Q(H+3L)Q(S)}p_1^*T_y(Z_0), 
\]
so that 
\[
p_{2*}\left(i_{2*}T_y(Y_2)\right)=\mathfrak{q}_1(E_1)p_1^*T_y(Z_0),
\]
where $\mathfrak{q}_1(E_1)$ is given by
\[
\mathfrak{q}_1(E_1)=\left(\sum_{n=1}^3C_n\mathfrak{q}_2(U_n)\right)\frac{Q(E_1)Q(H+2L-E_1)Q(H+3L-E_1)Q(S-E_1)}{Q(H+2L)Q(H+3L)Q(S)}.
\]

Now let $\rho$ denote the crepant resolution $p_1\circ p_2:Y_2\to Y_0$. Then
\[
i_{0*}T_y^{\text{str}}(Y_0)=\rho_*(i_{2*}T_y(Y_2))=p_{1*}p_{2*}(i_{2*}T_y(Y_2))=p_{1*}(\mathfrak{q}_1(E_1)p_1^*T_y(Z_0)),
\]
so that by the projection formula and Lemma~\ref{L2} we have
\[
i_{0*}T_y^{\text{str}}(Y_0)=\left(\sum_{n=1}^3D_n\mathfrak{q}_1(V_n)\right)T_y(Z_0),
\]
where $V_1=H+2L$, $V_2=H+3L$, $V_3=S$, and $D_n=\prod_{m \neq n} V_m/(V_m-V_n)$. Moreover, by B.5.8 in \cite{IT} we have
\[
T_y(Z_0)=Q(H)Q(H+2L)Q(H+3L)\pi^*T_y(B),
\]
where we recall $\pi$ denotes the projection $Z_0=\mathbb{P}(\mathscr{O}_B\oplus \mathscr{L}^2\oplus \mathscr{L}^3)\to B$. As such, for the SO(5) model $Y\to B$ we have
\begin{equation}\label{FF}
\iota_*T_y^{\text{str}}(Y)=\left(\sum_{n=1}^3D_n\mathfrak{q}_1(V_n)\right)Q(H)Q(H+2L)Q(H+3L)\pi^*T_y(B),
\end{equation}
where $\iota:Y\hookrightarrow \mathbb{P}(\mathscr{O}_B\oplus \mathscr{L}^2\oplus \mathscr{L}^3)$ denotes the inclusion.

\section{Generating functions for $\chi_y^{\rm{str}}$}\label{GF}
Let $\varphi:Y\to B$ be a Weierstrass fibration over an arbitrary smooth compact variety $B$, so that we have the following diagram
\begin{equation}\label{efd}
\xymatrix{
Y \ar[d]_{\varphi} \ar[r]^{\iota} & \mathbb{P}(\mathscr{E}) \ar[dl]^{\pi} &\\
B, &
}
\end{equation}
where $\mathscr{E}=\mathscr{O}_B\oplus \mathscr{L}^2\oplus \mathscr{L}^3$. The stringy $\chi_y$-characteristic is then given by
\begin{equation}\label{chiy}
\chi_y^{\text{str}}(Y)=\int_Y T_y^{\text{str}}(Y),
\end{equation}
where $\int$ denotes taking the degree of the zero-dimensional component of the class $T_y^{\text{str}}(Y)$. A consequence of the functoriality of the proper pushforward associated with proper maps is that the degree of a zero-dimensional class is invariant under proper pushforwards (see \cite{IT} Definition~1.4), so that
\begin{equation}\label{pf1}
\int_Y T_y^{\text{str}}(Y)=\int_B \varphi_*T_y^{\text{str}}(Y),
\end{equation}
where we recall $\varphi:Y\to B$ is the proper surjective map which endows $Y$ with the structure of an elliptic fibration. By equations \eqref{chiy} and \eqref{pf1} we then have
\[
\chi_y^{\text{str}}(Y)=\int_B \varphi_*T_y^{\text{str}}(Y),
\]
so that we can obtain a formula for $\chi_y^{\text{str}}(Y)$ in terms of invariants of $B$ by explicitly computing $\varphi_*T_y^{\text{str}}(Y)$. As in equation \eqref{FF} which represents the case $\mathcal{G}_Y=\text{SO}(5)$, for all Weierstrass fibrations we consider we have
\begin{equation}\label{shcf1}
\iota_*T_y^{\text{str}}(Y)=\mathfrak{q}_Y(L,S,H)\pi^*T_y(B),
\end{equation}
where $\mathfrak{q}_Y$ is a rational expression in $Q(L)$, $Q(S)$ and $Q(H)$ ($Q$ is the characteristic power series of $T_y$). By diagram \eqref{efd} we then have
\[
\varphi_*T_y^{\text{str}}(Y)=\pi_*\iota_*T_y^{\text{str}}(Y)=\pi_*(\mathfrak{q}_Y(L,S,H)\pi^*T_y(B))=\pi_*(\mathfrak{q}_Y(L,S,H))T_y(B),
\]
so that computing $\varphi_*T_y^{\text{str}}(Y)$ amounts to computing $\pi_*\mathfrak{q}(L,S,H)$. For this, we use Theorem~4.1 in \cite{ECRV}, which is a general formula for $\pi_*$.\footnote{A Maple implementation for computing $\pi_*$ may be found under the program name `PiStar' in www.math.fsu.edu/$\sim$hoeij/files/SHC/MapleCode} In particular, for $\mathcal{G}_Y=\text{SU}(2)$ we have
\[
\pi_*\mathfrak{q}_Y(L,S,H)= 1-2y+\frac{y+1}{y+s} \left(y+sl\frac{y^2s-(l^4+(s-1)l^2-s)y-l^4}{l^6+s^2y} \right),
\]
where $l=e^{(1+y)L}$ and $s=e^{(1+y)S}$.

As in the $\text{SU}(2)$ case, for a general Weierstrass fibration $\varphi:Y\to B$ as in Table~1, the pushforward  $\pi_*\mathfrak{q}_Y(L,S,H)$ (as in equation \eqref{shcf1}) will be a rational expression in $l=e^{(1+y)L}$ and $s=e^{(1+y)S}$ which we denote by $\mathcal{Q}_Y(l,s)$, so that
\[
\varphi_*T_y^{\text{str}}(Y)=\mathcal{Q}_Y(l,s)T_y(B).
\]
Now denote by $\mathcal{Q}_Y(t)$ the formal power series in $t$ obtained by evaluating $\mathcal{Q}_Y(l,s)$ at $(l',s')$, where $l'=e^{(1+y)tL}$ and $s'=e^{(1+y)tS}$, and given a formal power series $f(t)=b_0+b_1t+\cdots$, let $[t^k]f(t)=b_k$. It then follows that if $\text{dim}(B)=d$, then $\chi_y^{\text{str}}(Y)$ is given by
\[
\chi_y^{\text{str}}(Y)=[t^d]\left(\mathcal{Q}_Y(t)\prod_{i=1}^dQ(t\lambda_i)\right),
\]
where $\lambda_1,...,\lambda_d$ are the Chern roots of the tangent bundle of $B$ and $Q(z)$ is the characteristic power series of $T_y$, which is given by
\begin{equation}\label{cps}
Q(z)=\frac{z(1+y)}{1-e^{-z(1+y)}}-yz.
\end{equation}
Now let $\chi^{\text{str}}_y(Y;t)=\sum_k a_kt^k$ be the associated generating function for $\chi_y^{\text{str}}(Y)$ with respect to the dimension of $B$, so that
\begin{equation}\label{gscoeff}
[t^d]\chi^{\text{str}}_y(Y;t)=[t^d]\left(\mathcal{Q}_Y(t)\prod_{i=1}^dQ(t\lambda_i)\right).
\end{equation}

Now before deriving a closed form expression for $\chi^{\text{str}}_y(Y;t)$, we prove a preliminary lemma on power series. For this, we first need some definitions. Let $R$ be a commutative ring with 1, and consider the ring $R[[t]]$ of formal power series in the variable $t$ with coefficients in $R$. We then define the \emph{Hadamard product} in $R[[t]]$ as the map $\odot :R[[t]]\times R[[t]]\to R[[t]]$ given by
\[
\odot \left(\sum_{i=0}^{\infty}a_it^i,\sum_{j=0}^{\infty}b_jt^j\right)=\sum_{k=0}^{\infty}d_kt^k, \quad \text{where} \quad d_k=a_kb_k.
\]
For $f,g\in R[[t]]$ we will denote $\odot (f,g)$ by $f\odot g$. Given $\lambda_1,\ldots,\lambda_d \in R$, we use the notation $p_i := \lambda_1^i + \cdots \lambda_d^i$ and we let
\[ C= \prod_{i=1}^d (1 - \lambda_i t) = 1 - c_1 t + c_2 t^2 + \cdots +(-1)^dc_d,\]
so that $c_i$ is the $i$th {\em symmetric polynomial}, and $p_i$ is the $i$th {\em power polynomial} of $\lambda_1,\ldots,\lambda_d$.
\begin{lemma}
\label{l55}
Let $G(t) = a_0 + a_1 t + \cdots \in R[[t]]$.  Then
\[ \sum_{i=1}^d G(\lambda_i t) = G\odot (d + p_1 t + p_2 t^2 + \cdots) = d a_0 + G \odot (-tC'/C). \]
\end{lemma}
\begin{proof}
The first equality follows from the definition of the $p_i$. For the second, note that $-tC'/C$ is well defined because the polynomial $C$
has a constant term of 1.  The equation $-tC'/C = p_1 t + p_2 t^2 + \cdots$ is obvious for $d=1$ (geometric series). For $d>1$, recall that logarithmic
derivatives turn products into sums: $(CD)'/(CD) = C'/C + D'/D$.
\end{proof}

We now prove
\begin{theorem}\label{T77}
Let $\varphi:Y\to B$ be a Weierstrass fibration as in Table~1, and let $R(t)=\ln(Q(t))$, where $Q(t)$ is the characteristic power series of the normalized Hirzebruch class given by \eqref{cps}. Then the generating function $\chi_y^{\text{str}}(Y;t)$ is then given by
\[
\chi_y^{\text{str}}(Y;t)=\mathcal{Q}_Y(t)\exp\left(R(t)\odot \left(\frac{-tC'(t)}{C(t)}\right)\right),
\]
where $\mathcal{Q}_Y(t)$ is as given in equation \eqref{gscoeff}, and $C(t)=\sum_{i=0}^{\infty}(-1)^ic_it^i$, where the $c_i$ are formal variables representing the Chern classes of an arbitrary base $B$.
\end{theorem}
\begin{proof}
By definition, the generating series $\chi_y^{\text{str}}(Y;t)$ is given by
\[
\chi_y^{\text{str}}(Y;t)=\sum_{d=1}^{\infty} \chi_dt^d,
\]
where 
\[
\chi_d=[t^d]\left(\mathcal{Q}_Y(t)\prod_{i=1}^dQ(\lambda_it)\right),
\]
and the $\lambda_i$ are elements in $R$ representing the Chern roots of the tangent bundle of an unspecified base $B$ of dimension $d$. Now let $R(t)=\ln(Q(t))$. Then
\[
\prod_{i=1}^dQ(\lambda_it)=\exp\left(\sum_i^dR(\lambda_it)\right)=\exp\left(R(t)\odot \left(\frac{-tC'}{C}\right)\right),
\]
where the second equality follows from Lemma~\ref{l55} along with the fact that the constant term $a_0$ of $R(t)$ is given by $a_0=\ln(1)=0$. We then have
\begin{equation}\label{afd1}
\chi_d=[t^d]\left(\mathcal{Q}_Y(t)\exp\left(R(t)\odot \left(\frac{-tC'}{C}\right)\right)\right),
\end{equation}
and the only way that the RHS of \eqref{afd1} depends on $d$ is that the polynomial $C=\sum_i^d(-1)^ic_it^i$ is of degree $d$. As such, replacing $C$ in the RHS of \eqref{afd1} by the power series $C(t)=\sum_{i=1}^{\infty}(-1)^ic_it^i$ has no effect, since adding terms of higher degree greater than $d$ will play no role in taking the $d$th coefficient. And moreover, replacing the polynomial $C$ in the RHS of \eqref{afd1} by the power series $C(t)$ removes any dependence on the RHS of \eqref{afd1} on $d$, thus
\[
\chi_y^{\text{str}}(Y;t)=\mathcal{Q}_Y(t)\exp\left(R(t)\odot \left(\frac{-tC'(t)}{C(t)}\right)\right),
\]  
as desired.
\end{proof}
An immediate corollary of Theorem~\ref{T77} is then given by
\begin{corollary}
Let $\varphi:Y\to B$ be any of the Weierstrass fibrations as given un Table~1. Then under the notation and assumptions of Theorem~\ref{T77}, the power series $\chi_{-1}^{\text{str}}(Y;t)$ is a generating series for the stringy Euler characteristic $\chi_{\text{str}}(Y)$ of $Y$, i.e., the coefficient of $t^d$ in $\chi_{-1}^{\text{str}}(Y;t)$ yields a formula for $\chi_{\text{str}}(Y)$ over a base $B$ of dimension $d$.   
\end{corollary}

For an illustration of Theorem~\ref{T77}, consider the case when $\mathcal{G}_Y=\text{SO(6)}$ as in Table~1. In such a case, we have $S=2L$ so that $\mathcal{Q}_Y(l,s)=\mathcal{Q}_Y(l)$ (recall $l=e^{(1+y)L}$ and $s=e^{(1+y)S}$), which is given by
\[
\mathcal{Q}_Y(l)=1-5\,y+{\frac { \left( y+1 \right)  \left(  \left( l+4 \right) y-l \right) }{{l}^{2}+y}}.
\]
We then have
\[
\mathcal{Q}_Y(t)=1-5y+{\frac { \left( y+1 \right)  \left(  \left( e^{t(1+y)L}+4 \right) y-e^{t(1+y)L} \right) }{{e^{2t(1+y)L}}+y}},
\]
so that
\[
\chi_y^{\text{str}}(Y;t)=\left(1-5y+{\frac { \left( y+1 \right)  \left(  \left( e^{t(1+y)L}+4 \right) y-e^{t(1+y)L} \right) }{{e^{2t(1+y)L}}+y}}\right)\exp\left(R(t)\odot \left(\frac{-tC'(t)}{C(t)}\right)\right).
\] 
In particular, in the Calabi-Yau case (i.e., when $L=c_1$) we have
\[
\chi_y^{\text{str}}(Y;t)=a_1t+a_2t^2+a_3t^3+a_4t^4+\cdots,
\]
with
\[
a_1=(y^2-10y+1)c_1, \quad \quad a_2=-6(y-1)c_1^2y,
\]

\[
a_3=-1/12(48(y^2-4y+1)c_1^2y-(y^2-10y+1)^2c_2)c_1,
\]
and
\[
a_4=-1/2(y-1)(y^2-10y+1)(4c_1^3+2c_1c_2-c_3)c_1y,
\]
where we recall $c_i$ denotes the $i$th Chern class of the base $c_i(B)$. It then follows that the stringy Euler characteristic of $Y$ is given by $12c_1$, $-12c_1^2$, $12c_1c_2+24c_1^3$, and $12(c_1c_3-2c_1^2c_2-4c_1^4)$ over a base of dimension 1,2,3 and 4 respectively. 

For all other cases considered in Table~1, the coefficients of $\chi_y^{\text{str}}(Y;t)$ (for $L=c_1$) up to degree 6 are listed at www.math.fsu.edu/$\sim$hoeij/files/SHC/CalabiYau.
 
\section{Stringy Hirzebruch-Riemann-Roch for SU(2) fibrations}
Let $\varphi:Y\to B$ be the SU(2) fibration as given in Table~1. We then have
\[
\mathcal{Q}_Y(l,s) =  1-2y+\frac{y+1}{y+s} \left(y+sl\frac{y^2s-(l^4+(s-1)l^2-s)y-l^4}{l^6+s^2y} \right),
\]
and the generating series $\chi_y^{\text{str}}(Y;t)$ in the Calabi-Yau case (i.e., when $L=c_1$) is given by
\[
\chi_y^{\text{str}}(Y;t)=a_1t+a_2t^2+a_3t^3+\cdots,
\]
where
\[
a_1=(y^2-10y+1)c_1, \quad a_2=3y(1-y)(S^2-5Sc_1+10c_1^2),
\]
and
\[
a_3=(y^2-4y+1)(S^3y-14S^2c_1y+49Sc_1^2y-60c_1^3y)+1/12(y^2-10y+1)^2c_1c_2.
\]

We now derive relations for the stringy Hodge numbers $h_{\text{str}}^{p,q}(Y)$ in the case that  $Y$ is a Calabi-Yau 4-fold (so that the base is a 3-fold). In such a case, it can be shown that $c_1c_2=24$ (for example, by Theorem~A.2 of \cite{EFY}), so that $a_3$ simplifies to 
\[
a_3=(y^2-4y+1)(S^3y-14S^2c_1y+49Sc_1^2y-60c_1^3y)+2(y^2-10y+1).
\]
Writing $a_3$ as a polynomial in $Y$ we then have
\begin{equation}\label{scyg}
a_3=\chi(y)=2+\chi_1y+\chi_2y^2+\chi_3y^3+2y^4,
\end{equation}
where
\[
\chi_1=\chi_3=S^3+14S^2c_1+49Sc_1^2-60c_1^3 \quad\quad  \text{and} \quad\quad \chi_2=4\chi_1+204.
\]
The stringy Euler characteristic of $Y$ is then given by
\[
\chi_{\text{str}}(Y)=\chi(-1)=288+360c_1^3+6(14S^2c_1-S^3-49Sc_1^2).
\]

Now let $\rho:\widetilde{Y}\to Y$ be the crepant resolution of $Y$ given by the data in Table~2 (in fact, \emph{any} crepant resolution will suffice for our purposes), and recall that the $\chi_y$-genus of $\widetilde{Y}$ is given by
\[
\chi_y(\widetilde{Y})=\int_{\widetilde{Y}}T_y(\widetilde{Y}).
\]
Now since 
\[
\int_{\widetilde{Y}}T_y(\widetilde{Y})=\int_Y \varphi_*(T_y(\widetilde{Y}))=\chi_y^{\text{str}}(Y),
\]
we have
\begin{equation}\label{shrr}
\chi_y(\widetilde{Y})=\chi(y),
\end{equation}
where $\chi(y)=\chi_y^{\text{str}}(Y)$ is the stringy $\chi_y$-genus of $Y$ given by equation \eqref{scyg}. From the definition of stringy Hodge numbers we then have $h_{\text{str}}^{p,q}(Y)=h^{p,q}(\widetilde{Y})$, thus by equations \eqref{shrr} and \eqref{hnr} the coefficients $\chi_i$ of $\chi(y)$ yield linear relations among the stringy Hodge numbers of $Y$. In particular, we have $\chi_0=2$, reflecting the fact that $h_{\text{str}}^{p,0}(Y)=0$ for $0<p<4$ and $h_{\text{str}}^{0,0}(Y)=h_{\text{str}}^{4,0}(Y)=1$ (which follows from the fact that $Y$ is Calabi-Yau). From the coefficients $\chi_1$ and $\chi_2$ we have the relations
\[
\chi_1=h_{\text{str}}^{1,0}-h_{\text{str}}^{1,1}+h_{\text{str}}^{1,2}-h_{\text{str}}^{1,3}, 
\]
and
\[
\chi_2=h_{\text{str}}^{2,0}-h_{\text{str}}^{2,1}+h_{\text{str}}^{2,2}-h_{\text{str}}^{2,3}=h_{\text{str}}^{2,2}-2h_{\text{str}}^{1,2},
\]
which yields
\[
h_{\text{str}}^{1,3}=h_{\text{str}}^{1,2}-(S^3+14S^2c_1+49Sc_1^2-60c_1^3)-h_{\text{str}}^{1,1},
\]
and 
\[
h_{\text{str}}^{2,2}=2h_{\text{str}}^{1,2}+4(S^3+14S^2c_1+49Sc_1^2-60c_1^3)+204,
\]
where $h_{\text{str}}^{p,q}$ denotes $h_{\text{str}}^{p,q}(Y)$.

\begin{remark}
By the Shioda-Tate Wazir formula one can show that if $\psi:X\to B$ is a Calabi-Yau elliptic fibration then
\[
h^{1,1}(X)=h^{1,1}(B)+1+\text{rk}(\text{MW}(\psi))+\Gamma_{\psi},
\]
where $\text{rk}(\text{MW}(\psi))$ is the rank of the Mordell-Weil group of rational sections of $\psi:X\to B$ and $\Gamma_{\psi}$ is the number of irreducible fibral divisors not meeting the distinguished section of $\psi:X\to B$. For $\psi:X\to B$ any of the crepant resolutions of Weierstrass fibrations considered in this paper we have $\text{rk}(\text{MW}(\psi))=0$ and $\Gamma_{\psi}=n$, where $n$ is the number of blowups needed for the crepant resolution. For all Weierstrass fibrations $\varphi:Y\to B$ considered in this paper we then have
\[
h_{\text{str}}^{1,1}(Y)=h^{1,1}(B)+1+n,
\] 
where $n$ is the number of blowups needed for the crepant resolution of $Y$. As such, in the case of 4-folds once one of $h_{\text{str}}^{1,2}$, $h_{\text{str}}^{1,3}$ or $h_{\text{str}}^{2,2}$, is computed, the two remaining non-trivial stringy Hodge numbers may be determined by $\chi_y^{\text{str}}(Y)$ via stringy Hirzebruch-Riemann-Roch. 
\end{remark}
\section{Summary of results}
Given a Weierstrass fibration $\varphi:Y\to B$ with gauge group $\mathcal{G}_Y$ as in Table ~1, we computed $\varphi_*T_y(Y)$, which in each case is given by
\begin{equation}\label{Q}
\varphi_*T_y^{\text{str}}(Y)=\mathcal{Q}_Y(l,s)T_y(B),
\end{equation}
so that all the information coming from $Y$ is contained in the term $\mathcal{Q}_Y(l,s)$, which is a rational expression in $l=e^{(1+y)L}$ and $s=e^{(1+y)S}$, where $L$ is the first Chern class of the fundamental line bundle of $Y$ and $S$ is the divisor in $B$ over which the singularities of $Y$ are supported.\footnote{All computer calculations in this paper were performed with Maple, and all files relevant to our computations can be found at www.math.fsu.edu/$\sim$hoeij/files/SHC} From formula \eqref{Q} one can then give a formula for the stringy $\chi_y$-genus of $Y$ in terms of invariants of the base $B$, namely
\[
\chi_y^{\text{str}}(Y)=\int_B \mathcal{Q}_Y(l,s)T_y(B).
\] 

In particular, if the base $B$ is of dimension $d$, then 
\[
T_y(B)=\prod_{i=1}^d Q(\lambda_i),
\]
where the $\lambda_i$ are the Chern roots of the tangent bundle of $B$, and $Q(z)$ is the characteristic power series of the characteristic class $T_y$ (see equation \eqref{cps}). It then follows that the stringy $\chi_y$-genus is given by
\begin{equation}\label{if17}
\chi_y^{\text{str}}(Y)=[t^d]\left(\mathcal{Q}_Y\left(l',s'\right)\prod_{i=1}^d Q(t\lambda_i)\right),
\end{equation}
where $l'=e^{(1+y)tL}$, $s'=e^{(1+y)tS}$, and we recall $[t^d]$ picks out the $d$th coefficient of a formal power series in $t$. From a relative perspective, that is, from a perspective independent of the dimension of the base, we then derived a generating series $\chi_y^{\text{str}}(Y;t)$, where the coefficient of $t^d$ is precisely given by the RHS of \eqref{if17}. The generating series $\chi_y^{\text{str}}(Y;t)$ is then given by
\[
\chi_y^{\text{str}}(Y;t)=\mathcal{Q}_Y(t)\exp\left(R(t)\odot \left(\frac{-tC'(t)}{C(t)}\right)\right),
\]
where $\mathcal{Q}_Y(t)=\mathcal{Q}_Y\left(l',s'\right)$, $R(t)=\ln(Q(t))$, and $C(t)=\sum_i (-1)^ic_it^i$, so that the $c_i$s appearing in the coefficient of $t^d$ in $\chi_y^{\text{str}}(Y;t)$ represent the Chern classes of a base $B$ of dimension $d$. Since all $Y$ we consider admit crepant resolutions, the coefficient of $t^d$ in $\chi_y^{\text{str}}(Y;t)$ is in fact a polynomial in $y$ of the form 
\[
\chi(y)=\chi_0+\chi_1y+\chi_2y^2+\cdots+\chi_{d+1}y^{d+1},
\]
which is in fact (near) palindromic, so that $(-y)^{d+1}\chi(1/y)=\chi(y)$. The stringy Euler characteristic of $Y$ over a base of dimension $d$ is then given by $\chi_{\text{str}}(Y)=\chi(-1)$, and moreover, the coefficient $\chi_p$ in $\chi(y)$ yields linear relations among the stringy Hodge numbers $h_{\text{str}}^{p,q}(Y)$ for $Y$ over a base of dimension $d$, namely
\[
\chi_p=\sum_{q}(-1)^qh_{\text{str}}^{p,q}(Y).
\]

We conclude by listing $\mathcal{Q}_Y(l,s)$ for each $Y$ as listed in Table~1. 
\\
For $\mathcal{G}_Y=\text{SU}(2)$,

\[
\mathcal{Q}_Y(l,s) =
1 - 2y + \frac{y+1}{y+s} \left( y + sl\frac{(y+1)(s  y  -l^4)-y(s-1)l^2             }{ l^6+s^2y}\right).
\]
\\

For $\mathcal{G}_Y=\text{SU}(3), \text{USp}(4)$ and $G_2$,  

\[
\mathcal{Q}_{Y}(l,s) =
1 - 3y + \frac{y+1}{y+s} \left(2y + sl\frac{(y+1)(s^2y  -l^4)-y(s-1)l(l^2+s)        }{ l^6+s^3y}\right).
\] 
\\

For $\mathcal{G}_Y=\text{SU}(4), \text{Spin}(7)$,

\[ \mathcal{Q}_{Y}(l,s)=
1 - 4y + \frac{y+1}{y+s} \left(3y + sl\frac{(y+1)(s^5y^2-l^8)-y(s-1)l(s+l)(l^5+s^3y)}{(l^6+s^4y)(l^4+s^2y)}\right).
\]
\\

For $\mathcal{G}_Y=\text{SU}(5)$, $\mathcal{Q}_{Y}(l,s)=$

\[ 
1 - 6y + \frac{y+1}{y+s} \left(5y + \frac{y(s-1)(l^7 +s^5y) + sl(y+1)(s^4y^2-l^5-sly(l-1)(l^2+sl+s^2))}{(l+y)(l^6+s^5y)}\right).
\]
\\

For the SO(3), SO(5) and SO(6) cases we have $S=4L,2L$ and $2L$ respectively, so that $\mathcal{Q}_{Y}(l,s)=\mathcal{Q}_{Y}(l)$. As such, for $\mathcal{G}_Y=\text{SO}(3)$,

\[\mathcal{Q}_{Y}(l)=
1 - 2y + (y+1)  \frac{(1+l)y-l^3}{l^4+y},
\] 
for $\mathcal{G}_Y=\text{SO}(5)$,
\[\mathcal{Q}_{Y}(l)=
1 - 3y + (y+1)  \frac{(2-l)y-l  }{l^2+y} + \frac{2yl(y+1)^2}{(l^2+y)^2},
\] 
and for $\mathcal{G}_Y=\text{SO}(6)$,
\[\mathcal{Q}_{Y}(l)=
1 - 5y + (y+1)  \frac{(4+l)y-l  }{l^2+y}.
\] 

For $\mathcal{G}_Y=F_4$,

\[
\mathcal{Q}_{Y}(l,s)=
1 - 5y + \frac{y+1}{y+s} \left(4y + sl\frac{(y+1)(s^3y  -l^4)-2y(s-1)l^2s           }{ l^6+s^4y}\right).
\]
\\

For $\mathcal{G}_Y=E_6$,

\[
\mathcal{Q}_{Y}(l,s)=
1 - 8y + \frac{y+1}{y+s} \left(7y + \frac{y(s-1)(l^9 +s^7y) + sl(y+1)(s^3y  -l^4)(l^3+s^3y)}{      (l^3+s^2y)(l^6+s^5y)}\right).
\]
\\

For $\mathcal{G}_Y=E_7$,

\[\mathcal{Q}_{Y}(l,s)=
1 - 9y + \frac{y+1}{y+s} \left(8y + \frac{y(s-1)(l^{10}+s^8y) + sl(y+1)(s^7y^2-l^8)          }{      (l^4+s^3y)(l^6+s^5y)}\right).
\]
\\
And finally, for $\mathcal{G}_Y=E_8$,

\[ \mathcal{Q}_{Y}(l,s)=
1 -11y + \frac{y+1}{y+s} \left(10y+ \frac{y(s-1)(l^6 +s^5 ) + sl(y+1)(s^4y  -l^4)          }{                 l^6+s^5y }\right).
\]

\bibliographystyle{plain}



\end{document}